\documentclass[12pt,a4 paper]{article}
\usepackage{amssymb}
\usepackage{amsmath}
\usepackage{amsthm}
\usepackage{amsfonts}
\usepackage{amstext}
\usepackage{graphicx}
\usepackage{cite}
\usepackage{color}
\usepackage{indentfirst}

\newtheorem{Theorem}{Theorem}

\newtheorem{Remark}{Remark}

\def\o{\Omega}

\def\h{\widehat}
\def\R{\mathbf{R}}

\def\h{\mathbf{H}}

\def\calO{{\cal O}}

\newcommand{\dif}{\mathrm{d}}
\newcommand{\di}{\mathrm{div}\,}
\newcommand{\p}{\partial}

\def\C{\mathrm{curl}\,}
\begin{document}
\title{Approximation of 2D Euler Equations by the Second-Grade Fluid Equations with Dirichlet Boundary Conditions}


\author{  Milton C. Lopes Filho$^{\lowercase{a}}$,  Helena J. Nussenzveig Lopes$^{\lowercase{a}}$, \\ Edriss S. Titi$^b$,  Aibin Zang$^{\lowercase{d,a}}$}

\date{}

\maketitle

\begin{center}
$^a$Universidade Federal do Rio de Janeiro, Av. Athos da Silveira Ramos, 149\\Ilha do Fund\~ao, Rio de Janeiro, RJ 21941-909, BRAZIL,\\
$^b$ Dept. of Mathematics, Texas A\&M University, 3368 TAMU\\
College Station, TX 77843-3368,  USA \\ and \\
Dept. of Computer Science and Applied Mathematics, Weizmann Institute
of Science, Rehovot 76100, Israel, \\
$^d$ Dept. of Mathematics, Yichun University, Yichun, Jiangxi, 336000, P. R.China.\\
\end{center}

\begin{abstract}
The second-grade fluid equations are a model for viscoelastic fluids, with two parameters: $\alpha > 0$, corresponding to the elastic response, and  $\nu > 0$, corresponding to viscosity. Formally setting these parameters to $0$ reduces the equations to the incompressible Euler equations of ideal fluid flow. In this article we study the limits $\alpha, \nu \to 0$ of solutions of the second-grade fluid system,
in a smooth, bounded, two-dimensional domain with no-slip boundary conditions. This class of problems interpolates between the Euler-$\alpha$ model ($\nu = 0$), for which the authors recently proved convergence to the solution of the incompressible Euler equations, 
and the Navier-Stokes case ($\alpha = 0$), for which the vanishing viscosity limit is an important open problem. We prove three results. First, we establish convergence of the solutions of the second-grade model to those of the Euler equations provided $\nu = \mathcal{O}(\alpha^2)$, as $\alpha \to 0$, extending the main result in  \cite{lntz}. 
Second, we prove equivalence between convergence (of the second-grade fluid equations to the Euler equations) and vanishing of the energy dissipation in a suitably thin region near the boundary, in the asymptotic  regime $\nu = \mathcal{O}(\alpha^{6/5})$, $\nu/\alpha^2 \to \infty$ as $\alpha \to 0$. This amounts to a convergence criterion similar to the well-known Kato criterion for the vanishing viscosity limit of the Navier-Stokes equations to the Euler equations.
Finally, we obtain an extension of Kato's classical criterion to the second-grade fluid model, valid if $\alpha = \mathcal{O}(\nu^{3/2})$, as $\nu \to 0$. The proof of all these results relies on energy estimates and boundary correctors, following the original idea by Kato.


\end{abstract}
\numberwithin{equation}{section}

\numberwithin{equation}{section}

%
\section{Introduction}

%
%

The second-grade fluid model is governed by the system:
\begin{equation} \label{eq1}
\left\{
\begin{array}{cl}
\partial_t v-\nu\Delta u+( u\cdot\nabla)v+\sum_{j=1}^2v_j\nabla u_j+\nabla p=0, &\mbox{in}~~\Omega\times (0, T) \\[3mm]
\nabla\cdot u= 0, &\mbox{in}~~\Omega\times (0, T)\\[2mm]
v=(I-\alpha^2\Delta)u, &\mbox{in}~~\Omega\times (0, T),
\end{array}
\right.
\end{equation}
see, e.g., \cite{ce,Galdi}.

We consider system \eqref{eq1} in a two-dimensional simply-connected smooth bounded domain $\o\subset\mathbb{R}^2$,  subject to the no-slip Dirichlet boundary condition on $\p\o$, i.e.,
\begin{equation}\label{bdry1}
u=0,~\mbox{on}~\p\o\times (0,T).
\end{equation}

Formally, if we set $\alpha = 0$, system \eqref{eq1} becomes the Navier-Stokes system, because
the term $\sum_{j=1}^{2} v_j \nabla u_j$ becomes a gradient and it can be incorporated into the pressure. On the other hand, if we set $\nu=0$ instead, system \eqref{eq1} becomes the Euler-$\alpha$ system and
setting both $\alpha$ and $\nu$ to zero yields the incompressible Euler equations, which we write as:

\begin{equation} \label{eu}
\left\{\begin{aligned} \partial_t  \bar{u}+ \bar{u}\cdot\nabla \bar{u}+\nabla\bar{ p}=0, ~&\mbox{in}~~\Omega\times (0, T) \\[3mm]
\nabla\cdot  \bar{u}= 0, ~&\mbox{in}~~\Omega\times (0, T). \\[2mm]
\end{aligned}\right.
\end{equation}
In this article,  the Euler system, \eqref{eu}, is subject to the non-penetration boundary condition,
\begin{equation}\label{bdry2}
\bar{u}\cdot\hat{n}=0,~ \mbox{on} ~\p\Omega\times (0,T),
\end{equation}
  where $\hat{n}$ denotes the exterior unit normal vector to $\p\o.$

In a recent paper, \cite{lntz}, the authors proved that, under suitable smoothness assumptions, solutions of the Euler-$\alpha$ system converge to solutions of the Euler system as $\alpha \to 0$, despite the presence of a boundary layer. The analogous problem for the $\nu \to 0$ limit of the Navier-Stokes equations is an important open problem.  As we have seen, the second-grade fluids equation provides a natural family of problems which formally interpolates between these situations, aside from having independent interest, see \cite{ce}. The purpose of the present article is to examine the limit $\alpha,\nu \to 0$ of the second-grade fluids equations, in the hope of shedding light into the contrast between the vanishing $\alpha$ limit of Euler-$\alpha$ and the vanishing viscosity limit of the Navier-Stokes system in the presence of a boundary layer.

Our investigation of the limit $\alpha,\nu \to 0$ of solutions to \eqref{eq1} is expressed in three different results. First, we prove that if $\nu = \mathcal{O}(\alpha^2)$, under appropriate conditions on regularity and convergence of initial data, solutions of the second-grade fluid equations converge to solutions of the Euler system in $L^2$ in space, uniformly in time, as $\alpha\to 0$. This result is a natural extension of the main result in \cite{lntz}, and the condition $\nu = \mathcal{O}(\alpha^2)$ can be interpreted as a smallness condition on $\nu$, implying that the second-grade fluid equations behave as a small perturbation of the Euler-$\alpha$ system. The other results are Kato-type  criteria (cf. \cite{kato}) for convergence. First, we prove that, if $\alpha^2 << \nu = \mathcal{O}(\alpha^{6/5})$, convergence is equivalent to vanishing of the energy dissipation rate in a region near the boundary of width $\mathcal{O}(\alpha^3 \nu^{-3/2})$. Note that, with the condition imposed on $\nu$, this width is known to vanish as $\alpha \to 0$, but no rate can be asserted. In this result, the second-grade fluid equations are still a small perturbation of the Euler-$\alpha$ model, but the perturbation is larger, so that convergence is lost and only the sharp characterization of convergence is retained. For the last result, we assume $\alpha = \mathcal{O}(\nu^{3/2})$. In this case, we prove that convergence is equivalent to vanishing of the energy dissipation rate in a region of width $\mathcal{O}(\nu)$ near the boundary, which is precisely the result obtained by Kato for the Navier-Stokes system in \cite{kato}.  In contrast with the first two results, this last result imposes a smallness condition on $\alpha$, which can be interpreted as thinking of the second-grade fluids system as a small perturbation of the Navier-Stokes system. The result proved is, therefore, a natural extension of the original result by Kato \cite{kato}. We illustrate the regions of validity of the three results in Figure 1 below,
\begin{figure}[htbp]
	\centering
	\includegraphics[height=2.5in, width=2.5in]{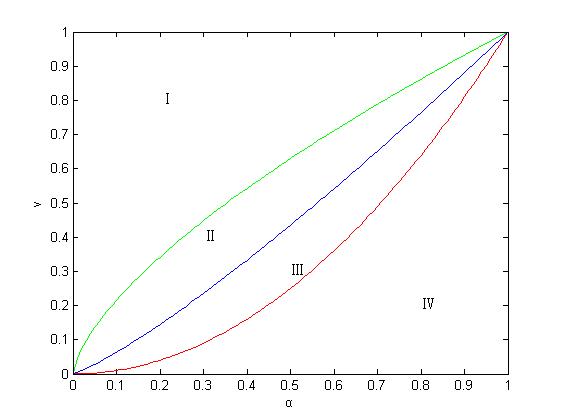}
	\caption{ \small Curve between region I and region II: $\nu=\alpha^{2/3}$; between II and III: $\nu=\alpha^{6/5}$; between III and IV: $\nu=\alpha^2$.}\label{fig:1}
\end{figure}
convergence in region IV, sharp convergence criteria in regions I and III, and no result obtained in region II.
The proofs of all three results are technically very similar, based on the use of energy estimates and boundary correctors, following the ideas introduced by Kato in \cite{kato}.

There is a large literature associated with the vanishing viscosity and vanishing $\alpha$ limits, which we briefly survey below. For the vanishing viscosity limit of the Navier-Stokes equations, convergence is known in cases without boundary, see  for example, \cite{Constantin1,Constantin2,ConstantinAMS,mas} and references therein, and for Navier boundary conditions, see \cite{CMR1, FLP, Xiao-Xin, mas2012, WXZ}. For no-slip boundary conditions the problem is open, with sharp convergence criteria obtained first by Kato in \cite{kato}, and reformulations of Kato's criteria obtained in \cite{kelliher,w}. For a recent survey on this subject, see \cite{Bardos-Titi}.
The literature associated with the $\alpha \to 0$ problem is much smaller. In addition to the convergence result in \cite{lntz}, already mentioned, it was shown in \cite{lt} that, in the whole space, solutions of Euler-$\alpha$  converge to the corresponding solutions  of the Euler equations, as $\alpha\to 0$. In \cite{BILN}, the limit $\alpha, \nu \to 0$ of the solutions of \eqref{eq1} with Navier-type boundary conditions was shown to converge to the corresponding solutions of the Euler equations, irrespective of the relative vanishing rates of $\alpha$ and $\nu$.

The remainder of this paper is organized in two sections. In section 2, we  introduce some basic notation, and present preliminary results.   In section 3, we state and prove our main results and draw some conclusions.

\section{Notations and preliminaries}

In this section, we introduce notation and present preliminary results concerning the second-grade fluid equations \eqref{eq1} and the Euler equations \eqref{eu} .

Let $\o \subseteq \R^2$ be a bounded, smooth, simply connected domain. We use the notation $H^{m}(\o)$  for the usual  $L^2$-based Sobolev spaces of order $m$, with the norm $\|\cdot\|_m$.  For the case $m=0$, $H^0(\o)=L^2(\o)$; we denote both norms by $\|\cdot\|$. By $\mathbf{H}^m(\o)$ we denote the Sobolev space of the vector fields $u=(u_1,u_2)$ such that $u_i\in H^m(\o)$, $i=1,2,$ and the norms in $\mathbf{H}^m(\o)$, $ (L^2(\o))^2$  are also denoted by $\|\cdot\|_m$, $\|\cdot\|$, respectively.
We denote by $C^{\infty}_c(\o)$ the space of smooth functions with infinitely many derivatives, compactly supported in $\o$, and  by $H^m_0(\o)$ the closure of $C^{\infty}_c(\o)$ under the $H^m$-norm.

We make use of the following function spaces.
\begin{equation*}
\begin{aligned}
&H=\{u\in (L^2(\Omega))^2: \di u=0~\mbox{in}~\Omega,~ u\cdot\hat{n}=0~\mbox{on}~\p\Omega\},\\
& V=\{u\in \mathbf{H}^1(\o):\di u=0~\mbox{in}~\Omega,~ u=0~\mbox{on}~\p\Omega\},\\
&W=\{u\in V: \C(u-\alpha^2 \Delta u)\in L^2(\o)\}.
\end{aligned}
\end{equation*}

 We will make frequent use of the identity:
\begin{equation}\label{zeroint}
\int_\o(\Psi\cdot\nabla) \Phi\cdot\Phi\dif x=0,
\end{equation}
 $~\mbox{for every}~\Psi\in \mathbf{H}^1(\o),~\mbox{with}~\Psi\cdot\hat{n}=0, \di \Psi=0, ~\mbox{and every}~\Phi\in \h^1(\o).$

We recall the two-dimensional   Ladyzhenskaya inequality (see, e.g.,\cite{cf, la}),
\begin{equation}\label{la}
\|\psi\|_{L^4(\o)}^2\le C\|\psi\|\|\psi\|_1,~\mbox{for every}~ \psi\in (H^1(\o))^2,
\end{equation}
where $C$ is a positive constant.

 Let $u =(u_1,u_2) \in V$, and set
\[\C u \equiv \partial_{x_1} u_2 - \partial_{x_2} u_1 = \nabla^{\perp} \cdot u,\] where
$\nabla^{\perp} = (-\partial_{x_2},\partial_{x_1})$. We   observe that
\begin{equation}\label{normu}
\|\C u\|=\|\nabla u\|, ~\mbox{for every} ~ u\in V.
\end{equation}

   Apply the $\C$  operator to the second-grade fluid equations \eqref{eq1} under the no-slip boundary conditions \eqref{bdry1} to  obtain the following  equivalent  two-dimensional  system,
\begin{equation}\label{eq2}
\left\{\begin{array}{cl} \partial_t q+\displaystyle{\frac{\nu}{\alpha^2}}(q-\C u)+u\cdot\nabla q=0 , &\mbox{in}~~\Omega\times (0, T) ,\\[3mm]
\nabla\cdot u= 0, &\mbox{in}~~\Omega\times (0, T), \\[2mm]
q=\C (u-\alpha^2\Delta u)&\mbox{in}~~\Omega\times (0, T), \\[2mm]
u=0&\mbox{on}~~\p\Omega\times (0,T).
\end{array}\right.
\end{equation}

Global well-posedness of  \eqref{eq1}, or equivalently, of (\ref{eq2}), has been established  in \cite{ce,Galdi}.  For the sake of completeness we state and prove the following:

\begin{Theorem}\label{2ndgradeth}
Let  $T>0$ be fixed, and $u^\alpha_0\in W.$ There exists a unique solution $u\in C([0,T]; V\cap W)$ of  problem (\ref{eq1})--\eqref{bdry1} (or (\ref{eq2})) with initial velocity $u_0^\alpha$, satisfying
\begin{equation}\label{est1}
\begin{aligned}
\|u(t)\|^2+\alpha^2\|\nabla u(t)\|^2+\nu\int_0^t\|\nabla u(\tau)\|^2\dif \tau=\|u_0^\alpha\|^2+\alpha^2\|\nabla u_0^\alpha\|^2,
\end{aligned}
\end{equation}
for every $t \in [0,T]$. Moreover,
\begin{equation}\label{est2}
\|q(t)\|^2\le e^{-\frac{1}{2}(\frac{\nu}{\alpha^2})t}\|q_0\|^2+ \frac{1}{2\alpha^2}(\|u_0^\alpha\|^2+\alpha^2\|\nabla u_0^\alpha\|^2),
\end{equation}
for every $t\in [0,T].$
\end{Theorem}
\begin{proof}
The existence and uniqueness of a solution to \eqref{eq1} was obtained in \cite{ce,Galdi}. By standard  energy estimates, it is easy to obtain identity \eqref{est1} as well as
\begin{equation}\label{est3}
\|q(t)\|^2\le e^{-\frac{1}{2}(\frac{\nu}{\alpha^2})t}\left(\|q_0\|^2+ \frac{\nu}{2\alpha^2}\int_0^t e^{\frac{1}{2}(\frac{\nu}{\alpha^2})s}\|\C u\|^2\dif s\right).
\end{equation}
From \eqref{normu}, \eqref{est1} and \eqref{est3} we conclude  \eqref{est2}.
\end{proof}

In the next section we will investigate the convergence of solutions of the second-grade fluid equations, as $\alpha \to 0$ and $\nu \to 0$, to the corresponding solutions of the 2D incompressible Euler equations. To this end, we will need the following existence and regularity result concerning the solution of  the Euler equations \eqref{eu} (see, for example, \cite{katolai,te1}).
\begin{Theorem}\label{kato-lai}
Fix $T>0$ and $s\ge 3$. Let $u_0\in \h^s(\o)\cap H$. Then there exists a unique solution $\bar{u}$ of  \eqref{eu}--\eqref{bdry2}, with initial velocity $u_0$, such that $\bar{u}\in C([0,T]; \h^s(\o))$. Moreover, $\bar{u}$ also belongs to $C^1([0,T];\h^{s-1}(\o))$ and $\|\bar{u}(t)\|=\|u_0\|$, for any $t \in [0,T]$.
\end{Theorem}
%
%
%

 \section {Main results}

In this section we state and prove three theorems concerning the limit as $\alpha,\nu \to 0$ of solutions of the second-grade fluid equations. These results aim at describing the second-grade equations as an interpolation between the Euler-$\alpha$ equations and the Navier-Stokes equations. It is hence natural to treat two different regimes, one in which $\nu$ is small with respect to $\alpha$ and the other in which $\alpha$ is small relative to $\nu$.

The first result consists of convergence to the corresponding Euler solution in the case $\nu = \mathcal{O}(\alpha^2)$, as $\alpha \to 0$.

 \begin{Theorem}\label{convtheorem}
Fix $T>0$ and  let $u_0\in \h^3(\Omega)\cap H$.  Assume   that we are given a family of  approximations $\{u_0^\alpha\}_{\alpha>0}\subset \h^3(\o)\cap V$ for $u_0$ satisfying:
\begin{equation}\label{IC}
\begin{aligned}
 &\|u^\alpha_0-u_0\|\to 0,\\
 &\alpha^2\|\nabla u^\alpha_0\|^2=o(1),\\
  &\|u^\alpha_0\|_{3}=\calO( \alpha^{-3}).
 \end{aligned}
 \end{equation}
Let $u^{\alpha,\nu} \in C([0,T];V \cap W)$ be the solution of \eqref{eq1}--\eqref{bdry1},
with initial velocity $u^\alpha_0$. Let $\bar{u} \in C([0,T];\h^3(\Omega) \cap H)
\cap C^1([0,T];\h^2(\Omega))$ be the  solution of the Euler equations \eqref{eu}--\eqref{bdry2}, with initial velocity $u_0$. Assume that  $\nu = \mathcal{O}(\alpha^2)$, as $\alpha \to 0$.
Then $u^{\alpha,\nu}$ converges to $\bar{u}$, strongly in $C([0,T];(L^2(\o))^2)$, as $ \alpha \to 0$.
\end{Theorem}

\begin{Remark}
In \cite{lntz} the authors introduced the notion of a {\it suitable family of approximations} to
$u_0 \in \h^3 \cap H$ as a family $\{u_0^{\alpha}\}$ satisfying \eqref{IC}.
In Proposition 1 of  \cite {lntz}, it was shown that, for any $u_0\in \h^1(\o)\cap H$ there exists a suitable family of approximations.
\end{Remark}

\begin{Remark}
Hereafter, $K$ will denote a positive constant which is independent of $\alpha$, but might depend on $u_0$.
\end{Remark}

\begin{proof}
We have, from Lemma 1 in \cite{ce}, that $\h^3(\o)\cap V \subseteq W$. Thus
from  (\ref{est1}) and \eqref{IC}, we deduce that for all $t\in [0,T]$,
\begin{equation}\label{enest}
\|u^{\alpha,\nu}(t)\|^2+\alpha^2\|\nabla u^{\alpha,\nu}(t)\|^2+\nu \int_0^t\|\nabla u^{\alpha,\nu}(\tau)\|^2\dif \tau=\|u^\alpha_0\|^2+\alpha^2\|\nabla u^\alpha_0\|^2\le K.
\end{equation}

From \eqref{normu}, together with \eqref{IC} we have:
\begin{equation}\label{iest}
\|q^\alpha_0\|\le\|\C u_0^\alpha \|+\alpha^2\|\C\Delta u_0^\alpha\|\le\frac{K}{\alpha}.
\end{equation}
By virtue of  \eqref{est2}, \eqref{iest} and \eqref{enest}, we have
\begin{equation}\label{qest}
\|q^\alpha(t)\|^2\le  \|q^\alpha_0\|^2+\frac{\|u_0^\alpha\|^2+\alpha^2\|\nabla u^{\alpha}_0\|^2}{2\alpha^2}\le \frac{K}{\alpha^2}.
\end{equation}
From \eqref{normu}, \eqref{eq2}, \eqref{enest} and \eqref{qest}  it follows that, for all $t\in [0,T]$, we have
\begin{equation}\label{curlu}
\alpha^2\|\C\Delta u^{\alpha,\nu} (t)\|\le\|q^\alpha(t)\|+\|\C u^{\alpha,\nu} (t)\|\le \frac{K}{\alpha}.
\end{equation}
We recall from Lemma 2 in \cite{lntz}, that, for all $\psi\in (H^3(\o))^2\cap V$,  we have
\begin{equation}\label{highest}
\|\psi\|_3\le K\| \C\Delta\psi \|.
\end{equation}
 Therefore, from \eqref{curlu} and \eqref{highest}, we conclude that  for all $t\in [0,T]$, it holds that
\begin{equation}\label{3rdest}
\|u^{\alpha,\nu}(t)\|_3\le \frac{K}{\alpha^3}.
\end{equation}

Following the notation introduced in \cite{lntz} we recall the boundary layer corrector, given by
\begin{equation}\label{24}
u_b= \nabla^{\perp}(z\bar{\psi}),
\end{equation} where $\bar{\psi}$ is the stream function associated to $\bar{u}$, and $z=z(x)$ is a cut-off function   supported in a $\delta-$neighborhood of the boundary, $\p\o$.

We list below some useful estimates on the boundary layer corrector obtained in \cite{lntz} (see also \cite{kato}).

For every $t\in [0,T]$, we have that:
%
\begin{alignat}{5}
& \|\p_t^{\ell}u_b(t)\|\le K\delta^{\frac{1}{2}}, \label{ub}\\
& \|\p^{\ell}_t\nabla u_b(t)\|\le K\delta^{-\frac{1}{2}},\label{gub}
\end{alignat}
where $\ell=0,1$,   and $K$  does not depend on $\delta$.

In what follows we will make use of the following  interpolation inequality (see \cite{ga2}, e.g.),
\begin{equation}\label{interinq} \|f\|^2_1\le K\|f\|\|f\|_2,~\mbox{for all}~f\in H^2(\o).
\end{equation}

Set $W^{\alpha,\nu}=u^{\alpha,\nu}-\bar{u}$. From \eqref{eq1} and \eqref{eu}, $W^{\alpha,\nu}$ satisfies:
\begin{equation}\label{eq3}
\left\{\begin{array}{cl} \partial_t W^{\alpha,\nu}+( u^{\alpha,\nu}\cdot\nabla) W^{\alpha,\nu}+(W^{\alpha,\nu}\cdot\nabla)\bar{u}\\
 =\di \tau^{\alpha,\nu}+\nabla\left(\bar{p}-p^\alpha-\displaystyle{\frac{|u^{\alpha,\nu}|^2}{2}}\right), &\mbox{in}~~\Omega\times (0, T) ,\\[3mm]
\di W^{\alpha,\nu}= 0, &\mbox{in}~~\Omega\times (0, T), \\[2mm]
W^{\alpha,\nu}\cdot\hat {n}=0&\mbox{on}~~\p\Omega\times (0, T), \\
W^{\alpha,\nu}(x,0)=u^{\alpha}_0-u_0 & \mbox{in}~~\Omega,
\end{array}\right.
\end{equation}
where\begin{equation*}
\begin{aligned}
&\di\tau^{\alpha,\nu}=\alpha^2\p_t\Delta u^{\alpha,\nu}+\nu\Delta u^{\alpha,\nu}+\alpha^2(u^{\alpha,\nu}\cdot\nabla)\Delta u^{\alpha,\nu}+\alpha^2\sum_{j=1}^2(\Delta u^{\alpha,\nu}_j)\nabla u_j^{\alpha,\nu}.
\end{aligned}
\end{equation*}
Multiply  equation (\ref{eq3}) by $W^{\alpha,\nu}$ and integrate over $\Omega\times (0,t)$, for $t\in [0,T]$. 
We then obtain
\begin{equation} \label{mainestW}
\begin{aligned}
\frac{1}{2}\|W^{\alpha,\nu}(t)\|^2+\int_0^t\int_\Omega (W^{\alpha,\nu}\cdot\nabla)\bar{u}\cdot W^{\alpha,\nu}\dif x\dif s=\int_0^t\int_{\Omega}\di \tau^{\alpha,\nu} \cdot W^{\alpha,\nu}\dif x\dif s
\\ + \frac{1}{2} \|W^{\alpha,\nu}(0)\|^2.
\end{aligned}
\end{equation}
Clearly,
\begin{equation}\label{mixterm}
\left|\int_0^t\int_\Omega (W^{\alpha,\nu}\cdot\nabla)\bar{u}\cdot W^{\alpha,\nu}\dif x\dif s\right|\le\|\nabla\bar{u}\|_{L^\infty(\Omega\times (0,T))}\int_0^t\|W^{\alpha,\nu}(s)\|^2\dif s.
\end{equation}
 Moreover,
\begin{equation}\label{eq4}
\begin{aligned}
\int_0^t\int_{\Omega}\di \tau^{\alpha,\nu}\cdot W^{\alpha,\nu}\dif x\dif s&=\alpha^2\int_0^t\int_\Omega\p_s\Delta u^{\alpha,\nu} \cdot W^{\alpha,\nu}\dif x\dif s\\
&-\alpha^2\int_0^t\int_\Omega(u^{\alpha,\nu}\cdot\nabla)\Delta u^{\alpha,\nu}\cdot \bar{u}\dif x\dif s\\
&-\alpha^2\int_0^t\int_\Omega\sum_{j=1}^2(\Delta u^{\alpha,\nu}_j)\nabla u_j^{\alpha,\nu}\cdot \bar{u}\dif x\dif s
\\
&+\nu\int_0^t\int_\Omega\Delta u^{\alpha,\nu}\cdot W^{\alpha,\nu}\dif x\dif s\\
&=:I_1(t)+I_2(t)+I_3(t)+I_4(t).
\end{aligned}
\end{equation}

Estimates for $I_1(t)$ and $I_2(t) + I_3(t)$ 
were already provided in the proof of Theorem 2 of \cite{lntz}. In particular, it was shown that,
for all $t\in [0,T]$,
\begin{equation}\label{est61}
\begin{aligned}
I_2(t) + I_3(t) \le K\alpha^2\int_0^t\|\nabla u^{\alpha,\nu}(s)\|^2\dif s+KT\alpha^2,
\end{aligned}
\end{equation}
see (4.18) and (4.19) in \cite{lntz}.

We will give some details of the estimate for $I_1$. We start by rewriting $I_1$, as in \cite{lntz}:

\begin{equation*}
\begin{aligned}
I_1(t) &=  \alpha^2 \int_0^t\int_\Omega \p_s\Delta u^{\alpha,\nu}\cdot W^{\alpha,\nu}\dif x\dif s\\
&=- \alpha^2 \int_0^t\int_{\Omega} \partial_s \nabla u^{\alpha,\nu} \,\cdot \,\nabla u^{\alpha,\nu} \dif x \dif s
 - \alpha^2 \int_0^t \int_{\Omega} \partial_s \Delta u^{\alpha,\nu} \cdot (\bar{u} - u_b) \dif x \dif s\\
 &- \alpha^2 \int_0^t \int_{\Omega} \partial_s \Delta u^{\alpha,\nu} \cdot u_b \dif x \dif s,
\end{aligned}
\end{equation*}
where $u_b$ is given in \eqref{24} for a suitable choice of $\delta$.

We have, for all $t\in [0,T],$
\begin{equation} \label{I1est}
I_1(t) \leq -\frac{\alpha^2}{4} \|\nabla u^{\alpha,\nu}(t)\|^2 +\frac{\alpha^2}{2} \|\nabla u^{\alpha}_0\|^2+ K \alpha^2\int_0^t \|\nabla u^{\alpha,\nu}\|^2 \dif s  + g(\delta,\alpha, u_0^{\alpha},u_0),
\end{equation}
with
\begin{align*}
g(\delta,\alpha,u_0^\alpha,u_0) &=  - \alpha^2\int_{\Omega} \nabla u_0^{\alpha}\cdot \nabla u_0 \dif x+ K\alpha^2 +
 K \alpha\delta^{-1/2}   + K \alpha^2\delta^{-1}+K\delta^{1/2}  + K\delta^{2/3},
 \end{align*}
see (4.15) of \cite{lntz} for details.
Choose $\delta=\delta(\alpha)$ such that
\begin{equation}\label{deltadata}
\delta(\alpha)\to 0~\mbox{and}~\frac{\alpha^2}{\delta(\alpha)}\to 0,~\mbox{as}~\alpha\to 0.
\end{equation}
It follows from our choice in \eqref{deltadata} and the hypotheses of  Theorem \ref{convtheorem}, namely \eqref{IC}, that
\begin{equation}\label{E2}
g(\delta,\alpha,u_0^\alpha,u_0) \to 0, \;\;\text{ as } \;\;\alpha \to 0.
\end{equation}

Finally, we  estimate the dissipative term. Here   we use the boundary corrector $u_b$ to allow integration by parts. We obtain
\begin{equation}\label{est9}
\begin{aligned}
I_4(t)&=\nu\int_0^t\int_\o\Delta u^{\alpha,\nu}\cdot(u^{\alpha,\nu}-\bar{u})\dif x\dif s=-\nu\int_0^t\int_\o(\nabla u^{\alpha,\nu} :\nabla u^{\alpha,\nu})\dif x\dif s\\
&-\nu\int_0^t\int_\o\Delta u^{\alpha,\nu}\cdot(\bar{u}-u_b))\dif x\dif s-\nu\int_0^t\int_\o\Delta u^{\alpha,\nu}\cdot u_b\dif x\dif s\\
&=-\nu\int_0^t\int_\o|\nabla u^{\alpha,\nu} |^2\dif x\dif s-\nu\int_0^t\int_\o(\nabla u^{\alpha,\nu}:\nabla(\bar{u}-u_b))\dif x\dif s\\
&-\nu\int_0^t\int_\o\Delta u^{\alpha,\nu}\cdot u_b\dif x\dif s\\
&\le-\nu\int_0^t\|\nabla u^{\alpha,\nu}(s)\|^2\dif s+\nu\int_0^t\|\nabla u^{\alpha,\nu}(s)\|\|\nabla\bar{u}(s)\|\dif s\\
&+\nu\int_0^t\|\nabla u^{\alpha,\nu}(s)\|\|\nabla u_b(s)\|\dif s+\nu\int_0^t\|\Delta u^{\alpha,\nu}(s)\|\|u_b(s)\|\dif s
\end{aligned}
\end{equation}
Thanks to \eqref{ub}, \eqref{gub} and the result in Theorem \ref{kato-lai}, we   obtain
\begin{equation}\label{i4est}
\begin{aligned}
I_4(t)&\le -\nu\int_0^t\|\nabla u^{\alpha,\nu}(s)\|^2\dif s+K\frac{\nu}{\alpha}\int_0^t(\alpha\|\nabla u^{\alpha,\nu}(s)\|)\dif s\\
&+K\frac{\nu}{\alpha}\delta^{-\frac{1}{2}}\int_0^t(\alpha\|\nabla u^{\alpha,\nu}(s)\|)\dif s+K\frac{\nu}{\alpha^2}\delta^\frac{1}{2}\int_0^t(\alpha^2\|\Delta u^{\alpha,\nu}(s)\|)\dif s.
\end{aligned}
\end{equation}
We apply  the interpolation inequality \eqref{interinq} to estimate the term $\|\Delta u^{\alpha,\nu}(s)\|$. Then   estimates \eqref{enest},  \eqref{3rdest} and \eqref{i4est} give
\begin{equation}\label{i4final}
I_4\le -\nu\int_0^t\|\nabla u^{\alpha,\nu}(s)\|^2\dif s+KT\frac{\nu}{\alpha}
+KT\frac{\nu}{\alpha}\delta^{-\frac{1}{2}}+KT\frac{\nu}{\alpha^2}\delta^\frac{1}{2}.
\end{equation}

Putting together  \eqref{est61}, \eqref{I1est}, \eqref{deltadata}, \eqref{E2} and (\ref{i4final}), we  conclude that, for all $t\in [0,T]$,
\begin{equation}\label{Cauchytensor}
\begin{aligned}
&\int_0^t\int_o\di\tau^{\alpha,\nu}\cdot W^{\alpha,\nu}\dif x\dif s=I_1(t)+I_2(t)+I_3(t)+I_4(t)\\
&\le -\frac{1}{4}\alpha^2\|\nabla u^{\alpha,\nu}(t)\|^2-\nu\int_0^t\|\nabla u^{\alpha,\nu}(s)\|^2\dif s+\frac{\alpha^2}{2}\|\nabla u^{\alpha,\nu}_0\|^2\\
&+ K\alpha^2\int_0^t\|\nabla u^{\alpha,\nu}(s)\|^2\dif s+g(\delta,\alpha,u_0^\alpha,u_0)\\
&+KT\alpha^2 +KT\frac{\nu}{\alpha}
+KT\frac{\nu}{\alpha}\delta^{-\frac{1}{2}}+KT\frac{\nu}{\alpha^2}\delta^\frac{1}{2}.
\end{aligned}
\end{equation}
 From \eqref{mixterm} and \eqref{Cauchytensor}, we have
\begin{equation}\label{est10}
\begin{aligned}
&\|W^{\alpha,\nu} (t)\|^2+\alpha^2\|\nabla u^{\alpha,\nu}\|^2 +2\nu\int_0^t\|\nabla u^{\alpha,\nu}(s)\|^2\dif s\le  K\int_0^t\|W^{\alpha,\nu}(s)\|^2\dif s \\
&+K(\|W^{\alpha,\nu} (0)\|^2+\alpha^2\|\nabla u^{\alpha,\nu}_0\|^2)+K\alpha^2\int_0^t\|\nabla u^{\alpha,\nu}(s)\|^2\dif s+\tilde{g}(u_0^\alpha,u_0,u_b(0)),
\end{aligned}
\end{equation}
where
\begin{equation*}
\begin{aligned}
\tilde{g}(u_0^\alpha,u_0,u_b(0))&=KT\alpha^2+KT\alpha\frac{\nu}{\alpha^2}\\
&+KT\frac{\nu}{\alpha^2}\alpha\delta^{-\frac{1}{2}}
+KT\frac{\nu}{\alpha^2}\delta^\frac{1}{2}+g(\delta,\alpha,u_0^\alpha, u_0).
\end{aligned}
\end{equation*}
From the conditions \eqref{deltadata},  \eqref{E2} and the assumption   $\nu = \mathcal{O}(\alpha^2)$, we
infer that
\begin{equation}\label{gg}
\tilde{g}(u_0^\alpha,u_0,u_b(0))\to 0,
\end{equation}
as $\alpha,\nu\to 0$.
Applying Gronwall's lemma to \eqref{est10},  we obtain
 \begin{equation}\label{west}
 \begin{aligned}
& \sup_{t\in[0,T]}\left(\|W^{\alpha,\nu}(t)\|^2+\alpha^2\|\nabla u^{\alpha,\nu}(t)\|^2\right)+\nu\int_0^T\|\nabla u^{\alpha,\nu}\|^2\dif t\\
 &\le e^{K_2T}\left[K_1(\|W^{\alpha,\nu}(0)\|^2+\alpha^2\|\nabla u^\alpha_0\|^2)+\tilde{g}( u_0^{\alpha},\bar{u}_0, u_b(0),\alpha)\right],
 \end{aligned}
 \end{equation}
 where $K_1,K_2$ do not depend on $\alpha,\nu$.
From \eqref{IC}, \eqref{gg} and \eqref{west}, it  follows that
\begin{equation}\label{stc}
\sup_{t\in (0,T)}(\|u^{\alpha,\nu}(t)-\bar{u}(t)\|^2+\alpha^2\|\nabla u^{\alpha,\nu}(t)\|^2)+\nu\int_0^T\|\nabla u^{\alpha,\nu}\|^2\dif t\to 0
\end{equation}
as $\alpha,\nu\to 0$, provided $\nu = \mathcal{O}(\alpha^2)$.
\end{proof}

The result we have just proved, Theorem \ref{convtheorem}, is a natural extension of Theorem 2
in \cite{lntz}, which treated the special case $\nu = 0, \alpha \to 0$. In fact, the proof we have just presented is an easy adaptation of the proof of Theorem 2 in \cite{lntz}.   It is natural to seek an extension of Kato's criterion, known for the case $\alpha = 0$, to the second-grade fluid equations.
We obtain two distinct results in this direction, one which is an extension of Theorem \ref{convtheorem}, with $\alpha^2 << \nu = \mathcal{O}(\alpha^{6/5})$, $\alpha \to 0$, and another which is an extension of Kato's original result in \cite{kato} to second-grade fluids, which works for $\alpha = \mathcal{O}(\nu^{3/2})$, as $\nu \to 0$.

Recall that a {\it suitable family of approximations} to a vector field $u_0\in \h^3(\Omega)\cap H$ is a family $\{u_0^\alpha\}_{\alpha>0}\subset \h^3(\o)\cap V$ satisfying \eqref{IC}.

\begin{Theorem}\label{katotheorem}
Fix $T>0$ and  let $u_0\in \h^3(\Omega)\cap H$. Let  $\{u_0^\alpha\}_{\alpha>0}\subset \h^3(\o)\cap V$, be a suitable family of approximations for $u_0$.  Let $u^{\alpha,\nu} \in C([0,T];V \cap W)$ be the solution of \eqref{eq1}--\eqref{bdry1}
with initial velocity $u^\alpha_0$. Let $\bar{u} \in C([0,T];\h^3(\Omega) \cap H)
\cap C^1([0,T];\h^2(\Omega))$ be the  solution of the Euler equations \eqref{eu}--\eqref{bdry2}, with initial velocity $u_0$. Assume that
\begin{equation} \label{kappathing}
\lim _{\alpha\to0}\frac{\nu}{\alpha^2}=\infty,
\end{equation}
and that
\begin{equation}\label{rel1}
{\nu} = \mathcal{O}({\alpha^{6/5}}).
\end{equation}
 Then
$u^{\alpha,\nu}$ converges  strongly to $\bar{u}$ in $C([0,T];(L^2(\o))^2)$,  as $\alpha\to 0$, if and only if

\begin{align} \label{kato1}
& \,\, \lim_{\alpha\to0 }\nu\int_0^T\int_{\o_\delta}|\nabla u^{\nu,\alpha}|^2\dif x\dif t=0,
\end{align}

where $\o_\delta$ is a $\delta-$neighborhood of $\p\o$ with
\begin{equation}\label{delta}
\delta=C\frac{\alpha^3}{\nu^\frac{3}{2}}.
 \end{equation}

 \end{Theorem}

\begin{proof}
Assume first that the family $u^{\alpha,\nu}$ converges, as $\alpha \to 0$, to $\bar{u}$ strongly  in $C([0,T];L^2(\o))$. Then, since $\|\bar{u}(t)\|$ is constant, $0\leq t \leq T$, and since we are under hypothesis \eqref{rel1}, it follows easily from the energy estimate \eqref{enest} together with the conditions \eqref{IC} for a suitable family of approximations, that the conclusion \eqref{kato1} holds true.

Conversely, we now suppose that condition \eqref{kato1} is valid, with $\delta$ given by \eqref{delta}. Let us denote the $L^2(0,T;L^2(\o_\delta))$-norm  by  $|\|\cdot\||$.

We use the notation in the proof of Theorem \ref{convtheorem}. Our starting point is the identity \eqref{mainestW}, for which, clearly, it suffices to estimate the terms $I_1$, $I_2$, $I_3$ and $I_4$, see \eqref{eq4}. The only estimate we need to modify is the estimate for $I_4$.

From \eqref{est9}, we have that
\begin{equation}\label{esti4}
\begin{aligned}
&I_4 \le-\nu\int_0^t\|\nabla u^{\alpha,\nu}(s)\|^2\dif s+KT\nu^\frac{1}{2}\\
&+\frac{\nu^\frac{1}{2}}{\delta^\frac{1}{2}}(\nu^\frac{1}{2}|\|\nabla u^{\alpha,\nu}\||)+\nu\delta|\|\Delta u^{\alpha,\nu}\||\\
&=-\nu\int_0^t\|\nabla u^{\alpha,\nu}(s)\|^2\dif s+KT\nu^\frac{1}{2}\\
&+\nu^\frac{1}{2}\delta^{-\frac{1}{2}}(\nu^\frac{1}{2}|\|\nabla u^{\alpha,\nu}\||)+\nu\delta^\frac{1}{2}\alpha^{-2}(\alpha^2|\|\Delta u^{\alpha,\nu}\||).
\end{aligned}
\end{equation}
From \eqref{3rdest} and \eqref{interinq}, we have
\begin{equation}\label{allint}
\begin{aligned}
\alpha|\|\nabla u\||=\frac{\alpha}{\nu^\frac{1}{2}}(\nu^\frac{1}{2}|\|\nabla u\||),~\alpha^2|\|\Delta u\||\le K\frac{\alpha^\frac{1}{2}}{\nu^\frac{1}{4}}(\nu^\frac{1}{2}|\|\nabla u\||)^\frac{1}{2}.
\end{aligned}
\end{equation}
From \eqref{allint}, we find that
\begin{equation}\label{l4bdd}
\begin{aligned}
I_4&\le -\nu\int_0^t\|\nabla u^{\alpha,\nu}(s)\|^2\dif s+KT\nu^\frac{1}{2}\\
&+\nu^\frac{1}{2}\delta^{-\frac{1}{2}}(\nu^\frac{1}{2}|\|\nabla u^{\alpha,\nu}\||)+K\delta^\frac{1}{2}\nu^\frac{3}{4}\alpha^{-\frac{3}{2}}(\nu^\frac{1}{2}|\|\nabla u^{\nu,\alpha}\||)^\frac{1}{2}.
\end{aligned}
\end{equation}

Recall that $\delta$ was given in \eqref{delta}. Then, in view of hypothesis \eqref{kappathing}, it follows that $\delta \to 0$ as $\alpha \to 0$. Furthermore, in light of our assumption on $\nu$ relative to $\alpha$, \eqref{rel1}, we have
\[\frac{\alpha^2}{\delta} = \mathcal{O}(\alpha^{4/5}),\]
so that both conditions in \eqref{deltadata} are satisfied as $\alpha \to 0$. Hence, as in Theorem \ref{convtheorem}, the conclusion \eqref{E2} follows.
Finally, we note that, due to hypothesis \eqref{rel1}, that $\nu = \mathcal{O}(\alpha^{6/5})$, we obtain
\[\nu^\frac{1}{2}\delta^{-\frac{1}{2}}=\nu^\frac{5}{4}\alpha^{-\frac{3}{2}}=\calO (1)\]
 as $\alpha\to 0$.

From \eqref{est61}, \eqref{I1est} and \eqref{l4bdd}, we have
\begin{equation}\label{estnew}
\begin{aligned}
&\|W^{\alpha,\nu} (t)\|^2+\alpha^2\|\nabla u^{\alpha,\nu}\|^2 +\nu\int_0^t\|\nabla u^{\alpha,\nu}(s)\|^2\dif s\le K(\|W^{\alpha,\nu} (0)\|^2 \\
&+ K\int_0^t\|W^{\alpha,\nu}(s)\|^2\dif s+\alpha^2\|\nabla u^{\alpha}_0\|^2)+K\alpha^2\int_0^t\|\nabla u^{\alpha,\nu}(s)\|^2\dif s+KT\nu^\frac{1}{2}\\
&+\nu^\frac{1}{2}\delta^{-\frac{1}{2}}(\nu^\frac{1}{2}|\|\nabla u^{\alpha,\nu}\||)+K\delta^\frac{1}{2}\nu^\frac{3}{4}\alpha^{-\frac{3}{2}}(\nu^\frac{1}{2}|\|\nabla u^{\nu,\alpha}\||)^\frac{1}{2}\\
&+KT\alpha^2+g(\delta,\alpha,u^\alpha_0, u_0),
\end{aligned}
\end{equation}
 where $g(\delta,\alpha,u^\alpha_0, u_0)$  is  as in \eqref{E2}. From \eqref{kato1}, \eqref{E2}, \eqref{estnew} and using the Gronwall lemma, we obtain \eqref{stc}, which is the convergence result we desired. This concludes the proof.

\end{proof}

In our final result we adopt the point of view that the second-grade equations are a perturbation of  the Navier-Stokes equations. Hence we consider the limit as $\nu \to 0$ under a smallness condition in  $\alpha$ relative to $\nu$.

\begin{Theorem} \label{katoNS}

Fix $T>0$ and  let $u_0\in \h^3(\Omega)\cap H$. Let  $\{u_0^\alpha\}_{\alpha>0}\subset \h^3(\o)\cap V$, be a suitable family of approximations for $u_0$.  Let $u^{\alpha,\nu} \in C([0,T];V \cap W)$ be the solution of \eqref{eq1}--\eqref{bdry1}
with initial velocity $u^\alpha_0$. Let $\bar{u} \in C([0,T];\h^3(\Omega) \cap H)
\cap C^1([0,T];\h^2(\Omega))$ be the  solution of the Euler equations \eqref{eu}--\eqref{bdry2}, with initial velocity $u_0$.  Assume that
\begin{equation} \label{rel2}
\alpha = \mathcal{O}(\nu^{3/2}), \;\;\;\mbox{ as } \;\;\;\nu \to 0.
\end{equation}
Then
$u^{\alpha,\nu}$ converges  strongly to $\bar{u}$ in $C([0,T];(L^2(\o))^2)$,  as $\nu \to 0$, if and only if
\begin{align}
& \,\, \lim_{\nu\to0}\nu\int_0^T\int_{\o_{C\nu}}|\nabla u^{\nu,\alpha}|^2\dif x\dif t=0,\label{kato11}
\end{align}
where $\o_{C\nu}$ is a $C\nu-$neighborhood of $\p\o.$
\end{Theorem}

\begin{proof}
As in Theorem \ref{katotheorem}, we first suppose that the family $u^{\alpha,\nu}$ converges, as $\alpha \to 0$, to $\bar{u}$ strongly  in $C([0,T];L^2(\o))$. Then, using the fact that $\|\bar{u}(t)\| = \|u_0\|$, $0\leq t \leq T$, it follows from the energy estimate \eqref{enest}, together with the conditions \eqref{IC} for a suitable family of approximations, that \eqref{kato11} holds true.

Next, we assume that \eqref{kato11} is valid.
Once again we use the notation in the proof of Theorem \ref{convtheorem}. We start by multiplying the equation for the difference $W^{\alpha,\nu} = u^{\alpha,\nu}-\bar{u}$, \eqref{eq3}, by $V^{\alpha,\nu}=u^{\alpha,\nu}-(\bar{u}-u_b)$; we then integrate the resulting identity on $\o\times (0,t)$, for all $t\in [0,T]$. We obtain
 \begin{equation}\label{correctw}
\begin{aligned}
&\frac{1}{2}\|u^{\alpha,\nu}(t)-\bar{u}(t)\|^2 - \frac{1}{2}\|u^{\alpha}_0-u_0\|^2 +\int_0^t\int_\Omega\p_s (u^{\alpha,\nu}-\bar{u}) u_b\dif x\dif s\\
&=\int_0^t\int_{\Omega}\di \tau^{\alpha,\nu} \cdot V^{\alpha,\nu}\dif x\dif s-\int_0^t\int_\Omega ((u^{\alpha,\nu}\cdot\nabla)u^{\alpha,\nu}-(\bar{u}\cdot\nabla)\bar{u})\cdot V^{\alpha,\nu}\dif x\dif s.
\end{aligned}
\end{equation}
Consider first the third term on the left-hand-side of \eqref{correctw}. We integrate by parts with respect to $t$ and we use \eqref{ub} and \eqref{enest} to deduce that
\begin{equation}\label{lhs}
\begin{aligned}
&\left|\int_0^t\int_\Omega\p_s (u^{\alpha,\nu}-\bar{u}) u_b\dif x\dif s\right|=\\
&\left|\int_\o(u^{\alpha,\nu}(x, t)-\bar{u}(x, t)) u_b(x,t)\dif x-\int_\o(u^{\alpha,\nu}(x, 0)-\bar{u}(x, 0)) u_b(x,t)\dif x\right.\\
&\left.-\int_0^t\int_\Omega (u^{\alpha,\nu}-\bar{u}) \p_su_b\dif x\dif s\right|\\
&\le \|u^{\alpha,\nu}(t)-\bar{u}(t)\|\|u_b(t)\|+\|u^{\alpha}_0-u_0\|\|u_b(0)\|\\
&+\int_0^t\|u^{\alpha,\nu}(s)-\bar{u}(s)\|\|u_b\|\dif s\\
&\le K\delta^\frac{1}{2}+K\|u^{\alpha}_0-u_0\|+KT\delta^\frac{1}{2}.
\end{aligned}
\end{equation}
 Using repeatedly identity \eqref{zeroint}  we obtain, for the second term on the right-hand-side  of \eqref{correctw},
 \begin{equation}\label{rhdsc}
 \begin{aligned}
 &\int_0^t\int_\Omega ((u^{\alpha,\nu}\cdot\nabla)u^{\alpha,\nu}-(\bar{u}\cdot\nabla)\bar{u})\cdot V^{\alpha,\nu}\dif x\dif s\le \int_0^t\int_\Omega |u^{\alpha,\nu}-\bar{u}|^2|\nabla\bar{u}|\dif s\dif x\\
 &+\int_0^t\int_\Omega ((u^{\alpha,\nu}\cdot\nabla)u^{\alpha,\nu}-(\bar{u}\cdot\nabla)\bar{u})\cdot u_b\dif x\dif s\\
 &\le K\int_0^t\|u^{\alpha,\nu}(s)-\bar{u}(s)\|^2\dif s+\int_0^t\int_\Omega |(u^{\alpha,\nu}\cdot\nabla)u^{\alpha,\nu}\cdot  u_b|+K\| u_b\|.
 \end{aligned}
 \end{equation}
 We make use of a version of the Poincar\'{e} inequality, valid for functions in $u \in H^1(\Omega_{\delta})$ such that $u = 0$ on $\partial\Omega$; see e.g. Lemma 3 in \cite{ILN08}. This inequality is given by
 \begin{equation}\label{pinq}
 \|u\|_{L^2(\o_\delta)}\le K\delta\|\nabla u\|_{L^2(\o_\delta)}.
 \end{equation}
 It follows from \eqref{rhdsc} and \eqref{ub}, together with \eqref{pinq}, that
 \begin{equation}\label{rhse}
 \begin{aligned}
 &\int_0^t\int_\Omega ((u^{\alpha,\nu}\cdot\nabla)u^{\alpha,\nu}-(\bar{u}\cdot\nabla)\bar{u})\cdot V^{\alpha,\nu}\dif x\dif s\le K\int_0^t\|u^{\alpha,\nu}(s)-\bar{u}(s)\|^2\dif s\\
 &+\int_0^t\|u^{\alpha,\nu}\|_{L^2(\o_\delta)}\|\nabla u^{\alpha,\nu}\|_{L^2(\o_\delta)} |u_b|_\infty\dif s+K\delta^\frac{1}{2}\\
 &\le K\int_0^t\|u^{\alpha,\nu}(s)-\bar{u}(s)\|^2\dif s+K\delta\int_0^t\|\nabla u^{\alpha,\nu}\|^2_{L^2(\o_\delta)}+K\delta^\frac{1}{2}.
 \end{aligned}
 \end{equation}
Next, we  will treat the first term on the right-hand-side of \eqref{correctw}, which we rewrite as
 \begin{equation}\label{rhsf}
\begin{aligned}
\int_0^t\int_{\Omega}\di \tau^{\alpha,\nu}\cdot V^{\alpha,\nu}\dif x\dif s&=-\int_0^t\int_\Omega\alpha^2\p_s\nabla u^{\alpha,\nu} : \nabla V^{\alpha,\nu}\dif x\dif s\\
&-\alpha^2\int_0^t\int_\Omega(u^{\alpha,\nu}\cdot\nabla)\Delta u^{\alpha,\nu}\cdot (\bar{u}-u_b)\dif x\dif s\\
&-\alpha^2\int_0^t\int_\Omega\sum_{j=1}^2(\Delta u^{\alpha,\nu}_j)\nabla u_j^{\alpha,\nu}\cdot (\bar{u}-u_b)\dif x\dif s
\\
&-\nu\int_0^t\int_\Omega\nabla u^{\alpha,\nu}: \nabla V^{\alpha,\nu}\dif x\dif s\\
&=:J_1(t)+J_2(t)+J_3(t)+J_4(t).
\end{aligned}
\end{equation}
Now we need to estimate the terms above one by one.
We begin with $J_1$, which we integrate by parts with respect to the time variable to obtain:
\begin{equation}\label{J1est}
\begin{aligned}
J_1&=-\int_0^t\int_\Omega\alpha^2\p_s\nabla u^{\alpha,\nu} : \nabla (u^{\alpha,\nu}-\bar{u}+u_b)\dif x\dif s\\
&=-\frac{1}{2}\alpha^2\|\nabla u^{\alpha,\nu}(t)\|^2+\frac{1}{2}\alpha^2\|\nabla u^{\alpha}_0\|^2\\
&-\alpha^2\int_0^t\int_{\Omega}\nabla u^{\alpha,\nu}\nabla\p_s\bar{u}\dif x\dif s+\alpha^2\int_{\Omega}\nabla u^{\alpha}_0 : \nabla u_0\dif x\\
&+\alpha^2\int_0^t\int_{\Omega}\nabla u^{\alpha,\nu} : \nabla\p_s u_b\dif x\dif s-\alpha^2\int_{\Omega}\nabla u^{\alpha}_0 : \nabla u_b(0)\dif x\\
&\le -\frac{1}{2}\alpha^2\|\nabla u^{\alpha,\nu}(t)\|^2+KT\alpha+KT\alpha\delta^{-\frac{1}{2}}+\bar{g}(u_0^\alpha,u_0,u_b(0)),
\end{aligned}
\end{equation}
where
\begin{equation}\label{g2form}
\begin{aligned}
\bar{g}(u_0^\alpha,u_0,u_b(0))&=\frac{\alpha^2}{2}\|\nabla u_0^\alpha\|^2+\alpha^2\int_{\Omega}(\nabla u^\alpha_0:\nabla u_0)\dif x-\alpha^2\int_{\Omega}(\nabla u^\alpha_0:\nabla u_b(0))\dif x.
\end{aligned}
\end{equation}

Next, we add $J_2$ and $J_3$. We find, easily, that

\begin{align}
&J_2(t)+J_3(t)=I_2(t)+I_3(t)\nonumber\\
&+\alpha^2\int_0^t\int_\Omega\left[(u^{\alpha,\nu}\cdot\nabla)\Delta u^{\alpha,\nu}u_b+\sum_{j=1}^2\Delta u^{\alpha,\nu}_j\nabla u_j^{\alpha,\nu}\cdot u_b\right]\dif x\dif s\nonumber\\
&\le K\alpha^2\int_0^t\|\nabla u^{\alpha,\nu}(s)\|^2\dif s+KT\alpha^2\nonumber\\
&+\alpha^2\int_0^t\int_\Omega\left[(u^{\alpha,\nu}\cdot\nabla)\Delta u^{\alpha,\nu}u_b+\sum_{j=1}^2\Delta u^{\alpha,\nu}_j\nabla u_j^{\alpha,\nu}\cdot u_b\right]\dif x\dif s,\label{J23est}
\end{align}
where $I_2$ and $I_3$ are defined in \eqref{eq4}. Using the H\"{o}lder inequality and \eqref{ub}, one can infer that
\begin{equation}\label{J4est}
\begin{aligned}
&J_4=-\nu\int_{0}^{t}\int_{\Omega}|\nabla u^{\alpha,\nu}|^2\dif x\dif s+\nu\int_{0}^{t}\int_{\Omega}\nabla u^{\alpha,\nu}\nabla\bar{u}\dif x\dif s\\
&-\nu\int_{0}^{t}\int_{\Omega}\nabla u^{\alpha,\nu}\nabla u_b\dif x\dif s\\
&\le-\nu\int_{0}^{t}\|\nabla u^{\alpha,\nu}\|^2\dif s+\nu\int_0^t\|\nabla u^{\alpha,\nu}\|\|\nabla \bar{u}\|\dif s\\
&+\nu\int_0^t\|\nabla u^{\alpha,\nu}\|_{L^2(\Omega_\delta)}\|\nabla u_b\|_{L^2(\Omega_\delta)}\dif s\\
&\le-\nu\int_{0}^{t}\|\nabla u^{\alpha,\nu}\|^2\dif s+KT\nu^{\frac{1}{2}}+\frac{\nu^{1/2}}{\delta^{1/2}}\left(\nu^{\frac{1}{2}}\int_0^t\|\nabla u^{\alpha,\nu}\|_{L^2(\Omega_\delta)}\right).
\end{aligned}
\end{equation}
From \eqref{3rdest}, \eqref{gub} and \eqref{pinq}, we deduce that
\begin{equation}\label{3rdcn}
\begin{aligned}
&\alpha^2\int_0^t\int_\Omega[(u^{\alpha,\nu}\cdot\nabla)\Delta u^{\alpha,\nu}]\cdot u_b \dif x\dif s=-\alpha^2\int_0^t\int_\Omega\Delta u^{\alpha,\nu}\cdot [ (u^{\alpha,\nu}\cdot\nabla) u_b]\dif x\dif s\\
&\le \alpha^2\|\nabla u_b\|_{L^{\infty}}\int_0^t\|u^{\alpha,\nu}(s)\|_{L^2(\o_\delta)}\|\Delta u^{\alpha,\nu}\|_{L^2(\o_\delta)}\dif s\\
&\le K\int_0^t\|\nabla u^{\alpha,\nu}(s)\|_{L^2(\o_\delta)}(\alpha^2\|\Delta u^{\alpha,\nu}\|_{L^2(\o_\delta)})\dif s.
\end{aligned}
\end{equation}
From \eqref{3rdcn} together with \eqref{allint} it follows that
\begin{equation}\label{3rdcn1}
\begin{aligned}
\left|\alpha^2\int_0^t\int_\Omega(u^{\alpha,\nu}\cdot\nabla)\Delta u^{\alpha,\nu}\cdot (u_b)\dif x\dif s\right|\le K\frac{\alpha^\frac{1}{2}}{\nu^\frac{3}{4}}(\nu^\frac{1}{2}|\|\nabla u^{\alpha,\nu}\||)^\frac{3}{2}.
\end{aligned}
\end{equation}
Similarly, it is easy to obtain
 \begin{equation}\label{3rdcn2}
\begin{aligned}
\left|\alpha^2\int_0^t\int_\Omega\sum_j^2\Delta u^{\alpha,\nu}_j\nabla u_j\cdot (u_b)\dif x\dif s\right|\le K\frac{\alpha^\frac{1}{2}}{\nu^\frac{3}{4}}(\nu^\frac{1}{2}|\|\nabla u^{\alpha,\nu}\||)^\frac{3}{2}.
\end{aligned}
\end{equation}
We put together \eqref{correctw} with the estimates in \eqref{lhs}, \eqref{rhse}, \eqref{J1est},\eqref{J23est}, \eqref{J4est}, \eqref{3rdcn1}, \eqref{3rdcn2} and \eqref{g2form}, and we find
\begin{equation}\label{diff1}
\begin{aligned}
&\|u^{\alpha,\nu} (t)-\bar{u}(t)\|^2+\alpha^2\|\nabla u^{\alpha,\nu}\|^2 +\nu\int_0^t\|\nabla u^{\alpha,\nu}(s)\|^2\dif s\\
&\le K(\|u^{\alpha} _0-u_0\|^2+ \|u^{\alpha}_0-u_0\|)\\
&+K\int_0^t\|u^{\alpha,\nu}(s)-\bar{u}(s)\|^2\dif s+K\delta\int_0^t\|\nabla u^{\alpha,\nu}\|^2_{L^2(\o_\delta)}+ K\alpha + K\delta^\frac{1}{2}\\
&+K\alpha^2\int_0^t\|\nabla u^{\alpha,\nu}(s)\|^2\dif s+K\alpha^\frac{1}{2}+K\alpha\delta^{-\frac{1}{2}}+K\nu^\frac{1}{2}\\
&+\frac{\nu^\frac{1}{2}}{\delta^\frac{1}{2}}(\nu^\frac{1}{2}|\|\nabla u^{\alpha,\nu}\||)+K\frac{\alpha^\frac{1}{2}}{\nu^\frac{3}{4}}(\nu^\frac{1}{2}|\|\nabla u^{\alpha,\nu}\||)^\frac{3}{2}+\bar{g}(u_0^\alpha,u_0,u_b(0)).
\end{aligned}
\end{equation}

We apply the  Gronwall  lemma to \eqref{diff1}, we take $\delta=C\nu$ and we use our hypothesis  \eqref{rel2} and\eqref{kato11},
to conclude that
\begin{equation*}\label{strconv}
\sup_{t\in (0,T)}(\|u^{\alpha,\nu}(t)-\bar{u}(t)\|^2+\alpha^2\|\nabla u^{\alpha,\nu}(t)\|^2)+\nu\int_0^T\|\nabla u^{\alpha,\nu}\|^2\dif t\to 0,
\end{equation*}
as $\nu\to 0.$
 \end{proof}

    We finish with a few concluding remarks. Our goal here was to probe the relation between the vanishing viscosity limit for the Navier-Stokes equations and the vanishing-$\alpha$ limit for the Euler-$\alpha$ system, exploring the diverse nature of the boundary layer for these problems by using the second-grade fluid system as an interpolant. What we found is that there appears to be a subtle and complicated change in behavior as viscosity and $\alpha$ vanish at different relative rates. Near Euler-$\alpha$ there is a region in $(\alpha,\nu)$ space where behavior similar to Euler-$\alpha$ is found, and further along, a region where vanishing viscosity limit is controlled by behavior of the fluid in a suitable thin region around the boundary.  In addition, near Navier-Stokes there is a region where the behavior appears similar to Navier-Stokes as well. There is also an intermediate region where we found no precise characterization of the Euler limit in the spirit of Kato's criterion. In this intermediate region we could formulate criteria for convergence in several ways, but we found no equivalence result.

What is the difference in the boundary layer problem for Euler-$\alpha$ and Navier-Stokes? Both situations are associated with a thin region of intense shear near the boundary, caused by discrepancy between the boundary conditions of the approximation and of the limit. In inviscid fluid flows, thin regions of intense shear are subject to the Kelvin-Helmholtz instability, which is the source of much of the difficulty in understanding boundary layers. 
Most likely, the mechanism of inhibiting Kelvin-Helmholtz instability by Euler-$\alpha$ and Navier-Stokes is quite different; understanding precisely how would be a very interesting topic for future investigation.


	It would be interesting to examine this problem from an asymptotic analysis point-of-view, examining the changing nature of the boundary layer equations for the different relative ways in which $\alpha$ and $\nu$ may vanish.  This could also lead to estimating the error terms in the
situations where convergence was established. Other natural open problems include requiring less smoothness from the underlying Euler solution, for example, looking at the case of bounded initial vorticity, higher dimensions and other $\alpha$ models.

\noindent{\bf Acknowledgements}

  E.S.T. is thankful to the  kind hospitality of the Universidade Federal do Rio de Janeiro (UFRJ) and Instituto Nacional de Matem\' {a}tica  Pura  e Aplicada (IMPA), where part of this work was completed.
The work of M.C.L.F. is partially supported by CNPq grant \# 303089 / 2010-5.
The work of H.J.N.L. is supported in part by CNPq grant \# 306331 / 2010-1 and FAPERJ grant \# E-26/103.197/2012.
The work of  E.S.T.  is supported in part by the NSF grants  DMS-1009950, DMS-1109640 and DMS-1109645. Also by CNPq-CsF grant \# 401615/2012-0, through the program Ci\^encia sem Fronteiras.
The work of A.B.Z. is supported in part by the CNPq-CsF grant \# 402694/2012-0, by the National Natural Science Foundation of China (11201411) and Jiangxi Provincial Natural Science Foundation of China (20122BAB211004), Higher Education  Teacher Training Foundation of Jiangxi Provincial Education Department  and Youth Innovation Group of Applied Mathematics  in Yichun University.

\end{document}